\documentclass{amsart}
\usepackage{hyperref}
\usepackage{amssymb}
\usepackage{amsrefs} 

\newtheorem{theorem}{Theorem}[section]
\newtheorem{lemma}[theorem]{Lemma}

\theoremstyle{remark}
\newtheorem{remark}[theorem]{Remark}

\numberwithin{equation}{section}

\newcommand{\ZZ}{{\mathbf{Z}}}

\newcommand{\col}{\ {:}\ }

\begin{document}

\title[Lower bounds on double-base representations]
{Lower bounds on the lengths\\
 of double-base representations}

\author[Dimitrov]{Vassil S. Dimitrov}
\address{
Center for Information Security and Cryptography,
University of Calgary,
2500 University Drive NW,
Calgary, AB T2N 1N4, Canada
}
\email{dimitrov@atips.ca}

\author[Howe]{Everett W. Howe}
\address{Center for Communications Research,
         4320 Westerra Court,
         San Diego, CA 92121-1967, USA}
\email{however@alumni.caltech.edu}
\urladdr{\url{http://www.alumni.caltech.edu/~however/}}


\date{1 February 2011}

\begin{abstract}
A \emph{double-base representation} of an integer $n$ is an expression
$n = n_1 + \cdots + n_r$, where the $n_i$ are (positive or negative) integers
that are divisible by no primes other than $2$ or $3$; the \emph{length} of the
representation is the number $r$ of terms.  It is known that there is a
constant $a >0$ such that every integer $n$ has a double-base representation of
length at most $a\log n / \log\log n$.  We show that there is a constant $c>0$
such that there are infinitely many integers $n$ whose shortest double-base
representations have length greater than 
$c\log n / (\log\log n \log\log\log n)$.  

Our methods allow us to find the smallest positive integers with no double-base
representations of several lengths.  In particular, we show that $103$ is the
smallest positive integer with no double-base representation of length $2$, 
that $4985$ is the smallest positive integer with no double-base representation
of length $3$, that $641687$ is the smallest positive integer with no 
double-base representation of length $4$, and that $326552783$ is the smallest
positive integer with no double-base representation of length $5$.
\end{abstract}

\maketitle

\section{Introduction}

A \emph{$\{2,3\}$-integer} is a positive or negative integer with no prime
divisors other than $2$ or $3$.  A \emph{length-$r$ double-base representation}
of an integer $n$ is an expression $n = n_1 + \cdots + n_r$, where each summand
is a $\{2,3\}$-integer.  Double-base representations were introduced by Imbert,
Mishra, and the first author~\cites{DIM2008, DIM2005} to help speed the
calculation of large multiples of points on elliptic curves.  (An earlier
version of the double-base number system, using only positive summands, was 
used by Jullien, Miller, and the first author~\cite{DJM1998} for modular 
exponentiation.)  In this paper we will be concerned not with the practical
applications of double-base representations
\cites{ADDS2006, BI2009, DI2006, MD2008, WLCL2006, ZZH2008}, but rather with
number-theoretic questions arising from their study.

The \emph{span} $s(n)$ of an integer $n$ is the smallest $r$ such that $n$ has
a double-base representation of length $r$.  Note that if an integer has a 
double-base representation of length $r>0$, then it also has a double-base 
representation of length $r+1$; to see this, simply replace a summand $n_i$ in 
a length-$r$ representation with the expression $3n_i - 2n_i$.  Thus, a 
positive integer has double-base representations of every length greater than
or equal to its span.

It is already known~\citelist{\cite{DJM1998}*{Thm.~4}\cite{DIM2008}*{Thm.~2}}
that $s(n) = O(\log n/\log\log n)$.  Our first result gives a lower bound for 
the growth of $s(n)$.

\begin{theorem}
\label{T:growth}
There is a constant $c>0$ such that for infinitely many values of $n$ we have 
\[s(n) > \frac{c\log n}{\log\log n\log\log\log n}.\]
\end{theorem}

For each $i > 0$, let $S(i)$ denote the smallest positive integer with 
span~$i$.  Clearly $S(1) = 1$ and $S(2) = 5$.  The fact that $S(3) = 103$, 
which can be proven through the use of a result of Ellison~\cite{Ellison}, is 
stated in~\cite{DIM2008}.  Furthermore, it is conjectured in~\cite{DIM2008} 
that $S(4) = 4985$.  Methods related to those we use in the proof of 
Theorem~\ref{T:growth} allow us to compute the values of $S(i)$ for $i\le 6$.

\begin{theorem}
\label{T:spans}
We have
\begin{align*}
S(3) &= 103,\\
S(4) &= 4985,\\
S(5) &= 641687,\\
S(6) &= 326552783.
\end{align*}
\end{theorem}

The proofs of both theorems rely on finding integers $m$ such that the 
reduction modulo $m$ of the set of $\{2,3\}$-integers is a very small subset 
of~$\ZZ/m\ZZ$.  We make this notion precise and prove Theorem~\ref{T:growth} 
in Section~\ref{S:growth}. In Section~\ref{S:primitive} we introduce the idea
of a \emph{doubly-primitive} double-base representation and we show that the
problem of determining whether an integer has a length-$r$ double-base 
representation is equivalent to determining whether certain other related
integers have doubly-primitive representations of various lengths. In 
Section~\ref{S:sparse} we find several moduli $m$ that are useful in producing
proofs that integers do not have doubly-primitive representations of lengths up
to $5$. Finally, in Section~\ref{S:spans} we use the moduli produced in 
Section~\ref{S:sparse} to prove Theorem~\ref{T:spans}.

Our proof of Theorem~\ref{T:spans} depends on computer calculations.  The 
programs we used, some written in Magma~\cite{magma} and some in C, can be 
found on the second author's web site: start at 
\begin{center}
\url{http://www.alumni.caltech.edu/~however/biblio.html}
\end{center}
and follow the links related to this paper.

In a recent paper~\cite{AHL2009}, {\'A}d{\'a}m, Hajdu, and Luca study the set
of integers that can be expressed in the form $a_1 s_1 + \cdots + a_k s_k$, 
where $k$ is fixed, the coefficients $a_i$ are integers taken from a fixed 
finite set~$A$, and the $s_i$ are integers that are $S$-units for some fixed 
finite set $S$ of primes.  In the special case where $S = \{2,3\}$ and 
$A = \{1,-1\}$, the set they study is exactly the set of integers of span at 
most~$k$.  They, like us, use a result of Erd\H{o}s, Pomerance, and 
Schmutz~\cite{EPS1991} to find auxiliary moduli~$m$ for which the image of the
set of $S$-units in $\ZZ/m\ZZ$ is small.  However, their interest lies in 
computing upper bounds for the density of such sets, while ours lies in 
producing lower bounds on the spans of individual integers --- and, of course, 
in computing actual values of~$s(n)$ and~$S(n)$.

\section{Proof of Theorem~\ref{T:growth}}
\label{S:growth}

For each positive integer $m$ we let $T(m)$ be the image in $\ZZ/m\ZZ$ of the
set of $\{2,3\}$-integers, and we let $t(m)$ denote the cardinality of $T(m)$.
For every $r\ge 2$ we define the \emph{expected degree-$r$ density} $D_r(m)$
of $m$ to be the smaller of $1$ and
\[ \frac{1}{m} \binom{t(m) + r - 1}{r}. \]

\begin{lemma}
\label{L:density}
Suppose $r\ge 2$ and $m$ is an integer whose expected degree-$r$ density is
less than $1$.  Then not every element of $\ZZ/m\ZZ$ can be expressed as a sum
of $r$ elements of $T(m)$.
\end{lemma}

\begin{proof}
There are $\binom{t(m) + r - 1}{r}$ ways of choosing $r$ elements from $T(m)$
with repetition, so there are at most this many sums of $r$ elements of~$T(m)$.
If $D_r(m)$ is less than~$1$, then the number of such sums is less than~$m$, so
some element of $\ZZ/m\ZZ$ is not such a sum.
\end{proof}

\begin{proof}[Proof of Theorem~\ref{T:growth}]
Recall that the Carmichael function $\lambda$ is the function that assigns to
each integer $m\ge 1$ the exponent of the multiplicative group $(\ZZ/m\ZZ)^*$.
Using~\cite{APR1983}*{Prop.~10, p.~201}, Erd\H{o}s, Pomerance, and Schmutz 
show~\cite{EPS1991}*{\S2} that there is a constant $d>0$ such that there are 
infinitely many squarefree $m$ such that
\begin{equation}
\label{EQ:EPS}
\lambda(m) < (\log m)^{d\log\log\log m}.
\end{equation}
We will prove Theorem~\ref{T:growth} with $c = 1/(3d)$.

We note for future reference that the function 
$\log x / (\log\log x \log\log\log x)$ is increasing for $x\ge 12006$, and that
every integer $n < 12006$ has a double-base representation of length $4$ or
less (see Section~\ref{S:spans}).

Let $m$ be one of the infinite number of squarefree integers that satisfy 
equation~\eqref{EQ:EPS} and for which we also have
\[ \lambda(m) \ge 4  \text{\quad and\quad}
    \frac{c\log m}{\log\log m \log\log\log m} \ge 4.\]
Since $m$ is squarefree, there are at most $\lambda(m) + 1$ distinct powers of
$2$ in $\ZZ/m\ZZ$ and at most $\lambda(m) + 1$ distinct powers of $3$ 
in~$\ZZ/m\ZZ$.  It follows that 
\[t(m) \le 2(\lambda(m)+1)^2 < \lambda(m)^3.\]
Let $r = \lfloor c \log m / (\log\log m \log\log\log m)\rfloor$, so that
$r \ge 4$.  Then 
\begin{align*}
D_r(m) & < t(m)^r / m\\
       & < \lambda(m)^{3r} / m\\
       & < e^{3rd\log\log m\log\log\log m} / m\\
       & \le e^{\log m} / m\\
       &= 1,
\end{align*}
so by Lemma~\ref{L:density} there is a nonnegative integer $n < m$ such that
the image of $n$ in $\ZZ/m\ZZ$ cannot be written as the sum of $r$ elements 
of~$T(m)$.  It follows that 
\[ s(n) \ge r + 1 > \frac{c \log m }{\log\log m \log\log\log m}
                  > \frac{c \log n }{\log\log n \log\log\log n}.\]
(The final inequality depends on the fact that $n\ge 12006$, but we know that
$n\ge 12006$ because the span of $n$ is at least~$r+1\ge 5$.)
\end{proof}

\section{Primitive representations}
\label{S:primitive}

Let $n = n_1 + \cdots + n_r$ be a length-$r$ double-base representation of an 
integer $n$.  We say that the representation is \emph{primitive} if the 
greatest common divisor of the $n_i$ is~$1$; we say that the representation is
\emph{doubly primitive} if one of the $n_i$ is not divisible by $2$ and a
different $n_i$ is not divisible by~$3$.  Note that a representation that is
primitive but not doubly primitive must have one summand equal to $\pm 1$, and
all of the other summands must be divisible by~$6$.

\begin{lemma}
\label{L:reps}
Let $r\ge 2$.  An integer $n$ has a double-base representation of length $r$
if and only if at least one of the following four statements holds\textup{:}
\begin{enumerate}
\item $n$ has a doubly primitive representation of length~$r$.
\item $n+1$ is divisible by $6$ and $(n+1)/6$ has a representation of 
      length $r-1$.
\item $n-1$ is divisible by $6$ and $(n-1)/6$ has a representation of 
      length $r-1$.
\item There is a $\{2,3\}$-integer $d > 1$ that divides $n$ such that
      $n/d$ has a primitive representation of length~$r$.
\end{enumerate}
\end{lemma}

\begin{proof}
Suppose $n$ has a length-$r$ representation $n = n_1 + \cdots + n_r$.  If the
representation is doubly primitive, then statement (1) holds.  If the
representation is primitive but not doubly primitive, then either 
statement~(2) or statement~(3) holds.  If the representation is not primitive,
then statement~(4) holds, where we take $d$ to be the greatest common divisor
of the~$n_i$.

The converse is clear.
\end{proof}

\begin{remark}
It is easy, of course, to determine whether a positive integer has a length-$1$
double-base representation.
\end{remark}

By recursion, we see that a proof that an integer $n$ has no length-$r$
representation can be constructed from proofs that a number of integers (no
larger than $n$) have no doubly primitive representations of certain lengths.
Thus, in the following sections we will mostly focus on doubly-primitive 
representations.

\section{Low-density moduli}
\label{S:sparse}

To show that an integer $n$ has no double-base representations of length~$r$, 
we can find a modulus $m$ with $D_r(m) < 1$ and hope that $n$ reduces to one of
the elements of $\ZZ/m\ZZ$ that cannot be written as a sum of $r$ elements 
of~$T(m)$.   As a practical matter, however, it is easier to use 
Lemma~\ref{L:reps} to reduce the problem to finding proofs that certain 
integers have no doubly-primitive representations; as we shall see, this allows
us to use smaller moduli~$m$.

For each positive integer $m$, we let $T_2(m)$ and $T_3(m)$ denote the images in
$\ZZ/m\ZZ$ of the sets $\{\pm 2^i \col i\ge 0\}$ and $\{\pm 3^i \col i\ge 0\}$,
respectively, and we let $t_2(m)$ and $t_3(m)$ denote the cardinalities of
these two sets. 
We have already defined what a doubly-primitive representation of an integer
is; now we define a \emph{doubly-primitive length-$r$ representation} of an
element $n$ of $\ZZ/m\ZZ$ to be an expression 
\[n = n_1 + \cdots + n_r,\]
where $n_1$ lies in $T_2(m)$, $n_2$ lies in $T_3(m)$, and all $n_i$ lie
in~$T(m)$.  For every $r\ge 2$ we define the 
\emph{expected doubly-primitive degree-$r$ density} $d_r(m)$ of $m$ to be the
smaller of $1$ and
\[\frac{t_2(m) t_3(m)}{m} \binom{t(m) + r - 3}{r - 2}.\]

We leave the proof of the following lemma to the reader; it is similar to that
of Lemma~\ref{L:density}.

\begin{lemma}
Suppose $r\ge 2$ and $m$ is an integer whose expected doubly-primitive
degree-$r$ density is less than~$1$.  Then not every element of $\ZZ/m\ZZ$ 
has a doubly-primitive length-$r$ representation.\qed
\end{lemma}

Suppose we have found an integer $m$ with $d_r(m)< 1$, and suppose we would 
like to prove that an integer $n$ has no doubly-primitive length-$r$ 
representation by showing that its image in $\ZZ/m\ZZ$ has no doubly-primitive
length-$r$ representation.  The next lemma, whose simple proof we omit,
suggests an efficient way of checking the latter condition.

\begin{lemma}
\label{L:algorithm}
Let $m>0$ and $r\ge 2$ be integers.  Define two subsets $S_1$ and $S_2$ of 
$\ZZ/m\ZZ$ as follows\textup{:}  If $r$ is odd, say $r = 2u + 1$, take
\begin{align*}
    S_1 &= \{ a_1 + a_2 + \cdots + a_{u+1} \col a_1 \in T_2(m),\ 
                                                a_2\in T_3(m),\\
        & \phantom{= \{ a_1 + a_2 + \cdots + a_{u+1} \col \qquad} 
                                            a_3,\ldots,a_{u+1}\in T(m) \}, \\
    S_2 &= \{ a_1 + a_2 + \cdots + a_u \col a_1,\ldots,a_u\in T(m) \};
\end{align*}
if $r$ is even, say $r = 2u$, take
\begin{align*}
    S_1 &= \{ a_1 + a_2 + \cdots + a_u \col a_1 \in T_2(m),\ 
                                            a_2,\ldots,a_u\in T(m) \}, \\
    S_2 &= \{ a_1 + a_2 + \cdots + a_u \col a_1 \in T_3(m),\ 
                                            a_2,\ldots,a_u\in T(m) \}.
\end{align*}
Then an element $n$ of $\ZZ/m\ZZ$ has no doubly-primitive length-$r$
representation if and only if the intersection 
$S_1 \cap \{n - a \col a \in S_2\}$ is empty.\qed
\end{lemma}

The work (and the amount of memory) required to compute the set intersection
mentioned in the lemma is a nearly linear function of the sum of the sizes of
the two sets $S_1$ and $S_2$.  We can work out reasonable approximations for
the sizes of these sets.  If $r=2u+1$ is odd, we have
\[
\#S_1 \le t_2(m) t_3(m)\binom{t(m) + u - 2}{u-1} 
      \approx \frac{t_2(m) t_3(m) t(m)^{u-1}}{(u-1)!}
\]
and
\[
\#S_2 \le              \dbinom{t(m) + u - 1}{u} 
      \approx \frac{             t(m)^u     }{ u   !};
\]
while if $r=2u$ is even, then
\[
\#S_1 \le t_2(m)       \binom{t(m) + u - 2}{u-1} 
      \approx \frac{t_2(m)        t(m)^{u-1}}{(u-1)!} 
\]
and 
\[
\#S_2 \le t_3(m)      \binom{t(m) + u - 2}{u-1} 
      \approx \frac{t_3(m)        t(m)^{u-1}}{(u-1)!}.
\]
Based on these upper bounds, we approximate $\#S1 + \#S2$ by the function
\[
w_r(m) = \begin{cases}
\big(u t_2(m) t_3(m) + t(m)\big) \dfrac{t(m)^{u-1}}{u!} & \text{if $r = 2u+1$ is odd,}\\
&\\
\big(t_2(m) + t_3(m)\big) \dfrac{t(m)^{u-1}}{(u-1)!}    & \text{if $r = 2u$ is even.}
\end{cases}
\]

Here is one method of finding values of $m$ with low densities and small work
estimates.  Given integers $a,b>0$, we write $2^a - 1 = 3^x u$ with $3\nmid u$
and $3^b - 1 = 2^y v$ with $2\nmid v$, and we take
$m = 2^y 3^x \gcd(u,v)$.  Then we compute the sets $T_2(m)$, $T_3(m)$, and 
$T(m)$, and compute $d_r(m)$ and $w_r(m)$ for the desired value of $r$.  
In Table~\ref{Tbl:densities} we list several useful values of $m$ that we 
obtained in this way, together with their densities and work factors for 
$r=2,3,4$, and~$5$.

\begin{table}
\caption{Expected doubly-primitive degree-$r$ densities $d_r(m)$ and work
factors $w_r(m)$ for $r=2,3,4,5$ and several values of $m$.}
\label{Tbl:densities} 
\vspace{1ex}
\begin{center}
\renewcommand{\arraystretch}{1.2}
\begin{tabular}{|r|r|r|r|r|r|r|}
\hline
$a$   & $b$   &                      $m$ & $\log d_2$ & $\log d_3$ & $\log d_4$ & $\log d_5$               \\ \cline{4-7}
      &       &                          & $\log w_2$ & $\log w_3$ & $\log w_4$ & $\log w_5$               \\ \hline\hline
$144$ & $432$ & $1811941545963463911360$ &   $-36.48$ &   $-24.70$ &   $-13.61$ &    $-2.93$               \\ \cline{4-7}  
      &       &                          &   $  7.06$ &   $ 12.88$ &   $ 18.84$ &    $24.47$               \\ \hline\hline
$288$ & $144$ &    $7409469211410651840$ &   $-31.39$ &   $-20.02$ &   $ -9.35$ &    $ 0.0\phantom{0} $    \\ \cline{4-7}
      &       &                          &   $  6.78$ &   $ 12.47$ &   $ 18.15$ &\multicolumn{1}{c|}{---}  \\ \hline\hline
$144$ & $144$ &      $38391032183474880$ &   $-26.80$ &   $-16.11$ &   $ -6.10$ &    $ 0.0\phantom{0} $    \\ \cline{4-7} 
      &       &                          &   $  6.39$ &   $ 11.79$ &   $ 17.08$ &\multicolumn{1}{c|}{---}  \\ \hline\hline
 $72$ & $216$ &        $952177069640160$ &   $-23.37$ &   $-13.61$ &   $ -4.55$ &    $ 0.0\phantom{0} $    \\ \cline{4-7}
      &       &                          &   $  6.38$ &   $ 11.35$ &   $ 16.14$ &\multicolumn{1}{c|}{---}  \\ \hline\hline
$144$ &  $48$ &         $54610287600960$ &   $-21.30$ &   $-11.67$ &   $ -2.72$ &    $ 0.0\phantom{0} $    \\ \cline{4-7}
      &       &                          &   $  6.00$ &   $ 10.73$ &   $ 15.63$ &\multicolumn{1}{c|}{---}  \\ \hline\hline
 $36$ & $108$ &          $1099511627760$ &   $-17.94$ &   $ -8.85$ &   $ -0.45$ &    $ 0.0\phantom{0} $    \\ \cline{4-7}
      &       &                          &   $  5.71$ &   $ 10.19$ &   $ 14.80$ &\multicolumn{1}{c|}{---}  \\ \hline
\end{tabular}
\end{center} 
\end{table}

\section{Proof of Theorem~\ref{T:spans}}
\label{S:spans}

Using the values of $m$ given in Table~\ref{Tbl:densities}, it is a simple
matter to use Lemmas~\ref{L:reps} and~\ref{L:algorithm} to construct proofs
that certain integers have spans greater than $2$, $3$, and~$4$.  Magma code
for doing this can be found in the file \texttt{DoubleBase.magma}, available at
the URL mentioned in the introduction.  The file also contains Magma code that
verifies that every positive integer less than $103$ has span at most~$2$, that
every positive integer less than $4985$ has span at most~$3$, and that every
positive integer less than $641687$ has span at most~$4$.  Together, these 
programs confirm the values for $S(3)$, $S(4)$, and $S(5)$ given in 
Theorem~\ref{T:spans}.

The proof that $S(6) = 326552783$ also depends on machine computation, but the
work involved is large enough that we move from Magma to C.

For the rest of this section, we set
\begin{align*}
n &= 326552783,\\
m &= 1811941545963463911360,\\
m_0 &= 4441033200890842920 = m / 408.
\end{align*}
Let $T$ be the set of all $\{2,3\}$-integers, and let $T_2$ and $T_3$ be the
subsets of $T$ consisting of the elements whose only prime divisors are $2$ 
and~$3$, respectively. We set
\begin{align*}
P &= \{ r + s \col r, s \in T \},\\
S &= \{ r + s + t \col r\in T_2, s\in T_3,  t\in T \}.
\end{align*}
Any number that can be written as the difference of an element of $P$ and an
element of $S$ has span at most $5$.  The program \texttt{SumOf5.c}, available
at the URL mentioned in the introduction, shows that every positive integer 
less than~$n$ is contained in $P-S$.  (In fact, the program only looks at 
elements of $P$ and $S$ that have absolute value at most $2n$ and that are 
obtained from summands of absolute value at most~$2^{61}$.)

To show that $n$ cannot be written as the sum of five $\{2,3\}$-integers, we 
use Lemmas~\ref{L:reps} and~\ref{L:algorithm}.  The programs found in 
\texttt{DoubleBase.magma} can carry out all the necessary computations, except
for the largest one: showing that $n$ has no doubly-primitive representations
of length~$5$.  To carry out this step of the proof, we would like to show that
the image of $n$ in $\ZZ/m\ZZ$ has no doubly-primitive representations of 
length~$5$.  This is a slightly awkward computation on current desktop 
computers, because the modulus $m$ does not fit into a $64$-bit word. Instead,
we find all doubly-primitive length-$5$ representations of the image of $n$
in~$\ZZ/m_0\ZZ$.  This computation is done by the program \texttt{Check5.c}.  
Then we take the resulting list of mod-$m_0$ representations and use the Magma
program found in the file \texttt{Check5.magma} to see whether any of them can
be lifted up to a mod-$m$ representation.  We find that none of the mod-$m_0$
representations can be lifted to a mod-$m$ representation, so $n$ has no 
doubly-primitive representations of length~$5$.


\section{Acknowledgment}
The authors are grateful to Igor Shparlinski for informing them of the
paper of {\'A}d{\'a}m, Hajdu, and Luca~\cite{AHL2009}.





\end{document}